\documentclass[a4paper, 12pt]{article}

\usepackage{amsmath}
\usepackage{amssymb}
\usepackage{amsfonts}
\usepackage{amsthm}
\usepackage{mathrsfs}
\usepackage[utf8]{inputenc}

\usepackage{comment}

\usepackage{authblk}

\theoremstyle{definition}

\newtheorem{example}{Example}[section]

\theoremstyle{plain}
\newtheorem{theorem}{Theorem}[section]

\newtheorem{corollary}{Corollary}[section]

\theoremstyle{remark}
\newtheorem{remark}{Remark}[section]
\newtheorem*{acknnowledgments}{Acknowledgments}
\usepackage{graphicx}

\newcommand{\tuple}[1]  {
	\left( #1 \right)
}

\newcommand{\catl}[1]  {
  \mathbb{#1}
}
\newcommand{\mor}[3]  {
  #1 \colon #2 \rightarrow #3
}

\newcommand{\id}[1]  {
	{\mathit{id}_{#1}}
}

\newcommand{\prob}[1]  {
	\overline{#1_n[x]}
}

\newcommand{\probc}  {
	\prob{\mathbf{C}}
}

\newcommand{\sprob}[1]  {	
	\mor{\id{\prob{#1}}}{\prob{#1}}{\prob{#1}}
}

\newcommand{\sprobr}  {	
	\sprob{\mathbf{R}}
}

\newcommand{\sprobra}  {	
	\sprob{\mathbf{A}}
}

\newcommand{\sprobrq}  {	
	\sprob{\mathbf{Q}}
}

\newcommand{\sprobc}  {	
	\sprob{\mathbf{C}}
}

\newcommand{\pprob}[1]  {
	#1_n[x]
}

\newcommand{\pprobr}  {
	\pprob{\mathbf{R}}
}

\newcommand{\pprobc}  {
	\pprob{\mathbf{C}}
}

\newcommand{\spprob}[1]  {	
	\mor{\id{\pprob{#1}}}{\pprob{#1}}{\pprob{#1}}
}

\newcommand{\spprobc}  {	
	\spprob{\mathbf{C}}
}

\author[1]{Micha{\l} R. Przyby{\l}ek}
\author[2]{Pawe{\l} Siedlecki}
\affil[1]{Institute of Informatics; 
Department of Mathematics, Informatics and Mechanics; 
University of Warsaw; Warsaw, Poland}
\affil[2]{Institute of Applied Mathematics and Mechanics; 
Department of Mathematics, Informatics and Mechanics; 
University of Warsaw; Warsaw, Poland}
{
    \makeatletter
    \renewcommand\AB@affilsepx{: \protect\Affilfont}
    \makeatother

    \affil[ ]{Emails}

    \makeatletter
    \renewcommand\AB@affilsepx{, \protect\Affilfont}
    \makeatother

    \affil[1]{mrp@mimuw.edu.pl}
    \affil[2]{psiedlecki@mimuw.edu.pl}
}

\title{A note on the complexity of a phaseless polynomial interpolation}
\date{\today}

\begin{document}
\sloppy

\maketitle

\abstract{In this paper we revisit the classical problem of polynomial interpolation, with a slight twist; namely, polynomial evaluations are available up to a group action of the unit circle on the complex plane. It turns out that this new setting allows for a phaseless recovery of a polynomial in a polynomial time. }

\section{Introduction}
Polynomial interpolation is a classical computational problem considered in numerical mathematics: given a set of points and values find a polynomial of a given degree 
that assumes given values at given points. A common interpretation is that a polynomial is fitted to some observational data. It is tacitly assumed that a data set 
is given \emph{exactly}. Hence if, for instance, we are fitting a polynomial to some measured observable, our data set faithfully represents measured values. However, in some 
applications (e.g., signal processing, speech recognition, complex quantities processing) our measurements do not  faithfully represent all characteristics of a measured signal, 
namely, they lack a \emph{phase}. 
This observation has led to the emergence of a new subject of \emph{phase retrieval}, which in short can be described as a study of what can be reconstructed if 
measurements are restricted to magnitudes while phases are lost. More on phase retrieval can be found, e.g., in \cite{phase-overview}. 

Recently the problem of phase retrieval has inspired a study of problems and algorithms using \emph{absolute value information} from the point of view of 
\emph{information-based complexity} (IBC). 
While most information-based complexity research concerned problems with information given as evaluations of linear functionals, 
\cite{avi} consider problems and algorithms that rely on information consisting of modules of real or complex valued 
linear functionals. 

In this paper we study \emph{phaseless} polynomial interpolation understood as an 
algorithmic task to construct a polynomial up to a phase, based on its values 
that are available up to a phase. This means that for a set of absolute values 
of evaluations of an unknown polynomial $f$, i.e.,  
$|f(x_1)|,\ldots,|f(x_k)|$ for some points $x_1,\ldots,x_k$, we wish to find $uf$ 
for some number $u$ with $|u|=1$. We analyze the relation between the degree 
of a polynomial $f$ and the computational effort needed to perform this 
\emph{phaseless interpolation}. We show that it is possible to do it in a 
polynomial time. It turns out that while the result for real polynomials is quite straightforward, the 
complex case is more subtle; for instance, we cannot rely on evaluations given at 
arbitrary distinct points. 

 The paper is organized as follows. The next section describes the computational framework for the problems of polynomial identification. Section~\ref{s:real} studies 
phaseless identification of real polynomials, whereas Section~\ref{s:complex} investigates phaseless identification of complex polynomials. We conclude the paper in Section~\ref{s:conclusion} and indicate areas of further work. 
 
\section{Computational framework and problem statement}

To fix the terminology let us recall some general notions of IBC. Since in this 
work we are 
dealing with problems that are meant to be solved exactly, we will not need much beside 
basic IBC notions.
We are interested in \emph{identification problems}, i.e., 
problems of the form: 
$$\mor{\id{V}}{V}{V}$$
and their \emph{solutions} that consist of an \emph{information operator} 
$N\colon V\to\mathbf{K}^*$ and an \emph{algorithm} 
$\phi\colon\mathbf{K}^*\to V$ such that $\phi\circ N=id_V$. 
Here $V$ is a set,
$\mathbf{K}$ is either the field of real numbers $\mathbf{R}$ or 
the field of complex numbers $\mathbf{C}$, 
and $\mathbf{K}^*$ is 
the set of all finite sequences of elements of $\mathbf{K}$. 
Let $\Lambda$ be some class of \emph{elementary information 
operations} that can act on elements of $V$, i.e., elements of $\Lambda$ are maps 
$V\to\mathbf{K}$.

Recall that, following common IBC convention, an information operator 
is a mapping 
$N\colon V\to\mathbf{K}^*$ such that for all $v\in V$: 
$$N(v)=\left(L_1(v),L_2(v;y_1),\ldots,L_{n(v)}(v;y_1,\ldots,y_{n(v)-1})\right),$$
where $L_1\in\Lambda$, $y_1=L_1(v)$, and for $i\in\{2,\ldots,n(v)\}$
$$L_i\colon V\times\mathbf{K}^*\to\mathbf{K}$$
$L_i(\cdot,\bar{y})\in\Lambda$ for all $\bar{y}\in\mathbf{K}^*$, and $y_i=L_i(v;y_1,\ldots,y_{i-1})$. 
At every step of a construction of $N(v)$ the decision whether 
to evaluate another elementary information operation or stop the computation of 
$N(v)$ is based only on the 
previous evaluations. Note that the number $n(v)$ of evaluations depends on $v$ via the 
sequence of intermediate values of elementary information operations $y_1,\ldots,y_k$.
More on information operators and their role in IBC in more general settings can 
be found, e.g., in \cite{tww1988} and \cite{plaskota1996noisy}. 

All objects that we will deal with can be described as finite sequences of real numbers as follows. We treat complex numbers as pairs of real numbers via the usual correspondence 
$z\leftrightarrow\tuple{\textnormal{Re}(z),\textnormal{Im}(z)}$.
Similarly, we treat real polynomials as finite sequences of reals via
$a_0+a_1x+\cdots+a_nx^n\leftrightarrow \tuple{a_0,a_1\ldots,a_n}$, 
and complex polynomials as finite sequences of reals via
$a_0+a_1x+\cdots+a_nx^n\leftrightarrow\tuple{\textnormal{Re}(a_0),
\textnormal{Im}(a_0),\textnormal{Re}(a_1),\textnormal{Im}(a_1),\ldots,
\textnormal{Re}(a_n),\textnormal{Im}(a_n)}$. For this reason, our algorithms essentially are partial functions 
$\phi\colon\mathbf{R}^*\to\mathbf{R}^*$. We choose for our computational model the real Blum-Shub-Smale machine augmented with the ability to compute $\sqrt{(\cdot)}$ and $\sin(\cdot)$ and assume that an algorithm has to be computable in this model. We will also, by a slight abuse of the term, call $\sqrt{(\cdot)}$ and $\sin(\cdot)$ arithmetic operations. 
More on Blum-Shub-Smale model of computation over the reals can be found in \cite{bss}.

We assume that all arithmetic operations on reals are allowed and are performed 
in a constant time. All described algorithms make use only of the possibility to 
store and access finite sequences of real numbers, perform arithmetic operations on them 
and do bounded iteration. Hence the cost of an algorithm is understood as the 
(maximal)
number of real arithmetic operations performed during its execution.

We define the computational cost of a solution $(N,\phi)$ as:
$$\mbox{cost}(N,\phi)=\sup_{v\in V}\left(\mbox{cost}_i(N,v)+ 
\mbox{cost}_a(\phi,N(v))\right)$$
where, for $v\in V$ and $y\in\mathbf{R}^*$:
\begin{align*}
\mbox{cost}_i(N,v)=&\ n(v)&\\
\mbox{cost}_a(\phi,y)=&\ \mbox{the number of arithmetic operations performed}&\\
&\ \mbox{by}\ \mbox{an algorithm}\ \phi\ \mbox{on input}\ y&
\end{align*}

The computational complexity of an identification problem $\mor{\id{V}}{V}{V}$
is defined as:
$$\mbox{comp}(\id{V})=\inf\{\mbox{cost}(N,\phi)\colon 
(N,\phi)\ \mbox{is a solution of}\ \id{V}\}$$

\vskip 1pc
 
For a real or complex 
linear space $V$, let $\equiv$ be a binary relation on $V$ defined as: 
$$x\equiv y\ \ \mbox{iff}\ \ x=u y\ \mbox{for some number}\ u\ 
\mbox{with}\ |u|=1$$ 
Clearly, $\equiv$ is an equivalence relation. We will write $\overline{V}$ for $V/\equiv$.

Suppose that $\Lambda$ is some class of elementary information operations consisting 
of linear functionals.
We define a new class of information operations:
$$|\Lambda|=\{|L|\colon L\in\Lambda\}$$
where $|L|(v)=|L(v)|$ for $v\in V$, here $|\cdot|$ is an absolute value in case of real 
spaces, and modulus for complex spaces. 
Note that elements of $|\Lambda|$ act on $\overline{V}$ in a natural way since 
$|L(uv)|=|L(v)|$ for all $v\in V$ and all numbers $u$ such that $|u|=1$.

We consider two identification problems related to  
the classes $\Lambda$ and $|\Lambda|$, correspondingly:
$$\id{V}\colon V\to V\ \ \mbox{and}\ \ 
\id{\overline{V}}\colon \overline{V}\to \overline{V}$$
The problem of identification of $v \in V$ based on values $L_1(v), L_2(v), \ldots, L_n(v)$ for some $L_i \in \Lambda$ is the exact identification with exact evaluations of $L_i$. 
This is the problem $id_V$ when the class $\Lambda$ is available.
In the setting of phase retrieval, one instead tries to identify $v \in V$ up to a unit $u$ (i.e.~a scalar $u$ such that $|u| = 1$) based on absolute values $|L_1(v)|, |L_2(v)|, \dotsc, |L_n(v)|$. This is the problem $id_{\overline{V}}$ when the class 
$|\Lambda|$ is available. 

If $L \in \Lambda$ is linear, we may fix a single $v$ from the orbit $\{uv \colon |u| = 1\}$ by fixing $L(v)$ from the orbit $\{uL(v) \colon |u| = 1\}$. Therefore, instead of considering $\id{\overline{V}}$ with available information 
$|\Lambda|$ we can equivalently consider the problem $\id{V}$, i.e., 
the exact identification of $v$ with a single exact evaluation $L(v)$ and a sequence of evaluations  $|L_1(v)|, |L_2(v)|, \dotsc, |L_n(v)|$. In the remainder we shall use the above reformulation of quotient problems.

\vskip 2pc

For any field $\mathbf{K}$, let: 
$$\pprob{\mathbf{K}}=\{a_0+a_1x+\cdots+a_nx^n\colon \ a_0,a_1,\ldots,a_n\in\mathbf{K}\}$$ 
be the linear space of all polynomials 
of degree at most $n$ ($n\in\mathbb{N}$) and coefficients from $\mathbf{K}$ , and let  
$\prob{\mathbf{K}}=\pprob{\mathbf{K}}/\equiv$. Note that for every 
$f\in\pprob{\mathbf{K}}$ and $s\in\mathbf{K}$ the evaluation of $f$ at $s$ is well 
defined as: 
$$L_{s}(f)=a_0+a_1s+\cdots+a_ns^n$$

We consider $\mathbf{K}$ being the field 
of real numbers, some subfield of it, or the field of complex numbers.
We assume that the domain of any polynomial 
from $\mathbf{K}_n[x]$ is: 
\begin{align*}
&\mathbf{R},\  \textnormal{if}\  \mathbf{K}\  \textnormal{is a subfield of}\ \mathbf{R}\\ 
&\mathbf{C},\  \textnormal{if}\  \mathbf{K}=\mathbf{C}
\end{align*}

We will consider two basic classes of elementary information operations acting on 
spaces of polynomials:
\begin{align*}
\Lambda^{\text{std}}=\ &\{f\mapsto f(x)\colon \ \mbox{for some}\ x\ 
\mbox{in the domain of}\ f\}&\\
|\Lambda^{\text{std}}|=\ &\{f\mapsto |f(x)|\colon \ \mbox{for some}\ x\ 
\mbox{in the domain of}\ f\}&
\end{align*}
In the sequel we will use algorithms that may rely on information 
operations from the class $\Lambda^{\text{std}}\cup |\Lambda^{\text{std}}|$, we will 
always make it explicit how many operations from each of the classes have been used 
by an algorithm. We will call the usage of an operation from $\Lambda^{\text{std}}$ 
an \emph{exact} evaluation, while the usage of an operation from $|\Lambda^{\text{std}}|$ 
will be called a \emph{phaseless} (or \emph{signless} in the real case) evaluation.

Our aim is to investigate the computational 
complexity of the problem: 
$$\sprob{\mathbf{K}}$$ 
when $|\Lambda^{\text{std}}|$ is available, or equivalently of the problem: 
$$\spprob{\mathbf{K}}$$ 
if we can use any number of information operations from $|\Lambda^{\text{std}}|$ and 
exactly one operation from $\Lambda^{\text{std}}$.

\section{Real polynomials}
\label{s:real}
This section deals with phaseless identification of real polynomials. We shall start with the following observation.
\begin{theorem}
\label{t:non-adaptreal}
The computational complexity of the problem 
$$\sprobr$$ 
for the class $|\Lambda^{\textnormal{std}}|$ is polynomial in $n$. 
Furthermore, it has a solution $(N,\phi)$ such that $N$ is nonadaptive and $\textnormal{cost}_i(N,v)=2n+1$ for all 
$v\in \overline{\mathbf{R}_n[x]}$, and 
$\sup\{\textnormal{cost}_a(\phi,y)\colon y\in N(\overline{\mathbf{R}_n[x]})\}$ 
is polynomial in $n$.
\end{theorem}
\begin{proof}
Let us consider the general form of a real polynomial:
$$p(x) = a_0 + a_1x + a_2x^2 + \cdots + a_nx^n$$
where $a_0,a_1,\ldots,a_n\in\mathbf{R}$.
Suppose that $0 \leq j \leq n$ is the smallest index such that $a_j$ is non-zero. The square of the absolute value of $p$ is
\begin{align*}
|p(x)|^2 &= (x^j(a_j + a_{j+1}x + a_{j+2}x^2 + \cdots + a_nx^{n-j}))^2\\
&=  x^{2j}(A_0 + A_1x + A_2x^2 + \cdots + A_{2(n-j)}x^{2(n-j)}) \\
\end{align*}
for some $A_0,A_1,\ldots,A_{2(n-j)}\in\mathbf{R}$.
Because $|p(x)|^2$ is a real polynomial of degree at most $2n$ we may determine its coefficients $A_k$ by using its evaluations at (any) $2n+1$ distinct real points and 
performing polynomial interpolation. 
Furthermore, every evaluation of $|p(x)|^2$ is of the form $(L(p))^2$, where 
$L \in |\Lambda^{\textnormal{std}}|$.

Now, observe that:
\begin{align*}
A_k &= \left\{
  \begin{array}{ll}
    a_j^2 & \textnormal{if}\  k = 0  \\
    2a_ja_{j+k} + \sum_{i=1}^{k-1} a_{j+i}a_{j+k-i} & \textnormal{if}\  
    1 \leq k \leq 2(n-j)  \\
  \end{array} \right.
\end{align*}
Therefore, assuming $\tuple{A_k}_{0 \leq  k \leq 2(n-j)}$ are given, coefficients 
$a_m$ may be computed inductively:
\begin{align*}
a_{m} &= \left\{
  \begin{array}{ll}
    0 & \textnormal{if}\  m < j \\
    \pm \sqrt{A_0} & \textnormal{if}\  m = j  \\
    \frac{A_{m-j} - \sum_{i=1}^{m-j-1} a_{j+i}a_{m-i}}{2a_j} & \textnormal{if}\  
    j+1 \leq m \leq n \\
  \end{array} \right.
\end{align*}
Thus, we may recover the polynomial $p(x)$ up to a sign (i.e., the sign of $a_j$), which completes the proof.
\end{proof}


One may wonder if $2n+1$ information operations are needed. The answer is no, provided we are satisfied with algorithms with exponential cost that use an adaptive information.
\begin{theorem}
\label{t:adaptreal}
The problem 
$$\sprobr$$ 
has a solution $(N,\phi)$ such that $N$ is an adaptive information operator using information 
operations from the class $|\Lambda^\textnormal{std}|$ and such that $\textnormal{cost}_i(N,v)=n+2$ for all 
$v\in \overline{\mathbf{R}_n[x]}$, and
$\sup\{\textnormal{cost}_a(\phi,y)\colon y\in N(\overline{\mathbf{R}_n[x]})\}$ 
is exponential in $n$.
\end{theorem}
\begin{proof}
We will equivalently demonstrate that it is possible to identify a 
real polynomial by $n+1$ signless evaluations and one exact evaluation. 

Let us consider $p \in \pprobr$ and a sequence $x_0, x_1, \ldots, x_n \in \mathbf{R}$ of distinct real numbers. Denote by $\catl{B}_k = \{-1, 1\}^k$ the $k$-th power of the one-dimensional unit circle and by $\catl{B}_k^{(2)} = \{\tuple{\overline{x}, \overline{y}} \in \catl{B}_k \times \catl{B}_k \colon \overline{x} \neq \overline{y}\}$. For every tuple $\overline{b} \in \catl{B}_{n+1}$ there is exactly one polynomial $w_{\overline{b}} \in \pprobr$ such that $w_{\overline{b}}(x_i) = b_i|p(x_i)|$ for all $0 \leq i \leq n$. 
Therefore, there are exactly $2^{n+1}$ polynomials with values, up to a sign, matching 
those of $p$ in each of the points $x_0,x_1,\ldots,x_n$. 
Moreover, if $\overline{b} \neq \overline{b'}$ then the equation $w_{\overline{b}}(x) = w_{\overline{b'}}(x)$ has at most $n$ solutions. It follows that the set 
$$\catl{S} = \{x \in \mathbf{R} \colon \exists_{\tuple{\overline{b}, \overline{b'}}\in \catl{B}_{n+1}^{(2)}} \;  w_{\overline{b}}(x) = w_{\overline{b'}}(x)\}$$ 
is finite, thus does not exhaust the whole $\mathbf{R}$. Therefore, there exists $x \in \mathbf{R}$ that differentiates any pair of polynomials $w_{\overline{b}}$ and  $w_{\overline{b'}}$, and we may perform one exact evaluation at that $x$, which will uniquely identify $\overline{b}$ such that $w_{\overline{b}}=p$, and hence also $p$. Moreover, $x$ can be effectively found in exponential time by a naive algorithm that separates an interval from the set of all zeros of polynomials $w_{\overline{b}} - w_{\overline{b'}}$ computed to \emph{any} finite precision. 
\end{proof}
One may also observe that since $\catl{S}$ is of Lebesgue measure zero, any point randomly chosen from any non-degenerated interval will almost surely not belong to $\catl{S}$. 

\begin{example}
Consider $p \in \mathbf{R}_n[x]$ for $n=3$. Let us choose $\{1, 2, 3, 4\}$ for the set of four evaluation points and assume that $|p(x)| = 1$ for every $x \in \{1, 2, 3, 4\}$. Figure~\ref{f:real:example} shows all possible polynomials of degree at most $3$ that satisfy the signless evaluations.
\begin{figure}
  \centering
    \includegraphics[width=0.7\textwidth]{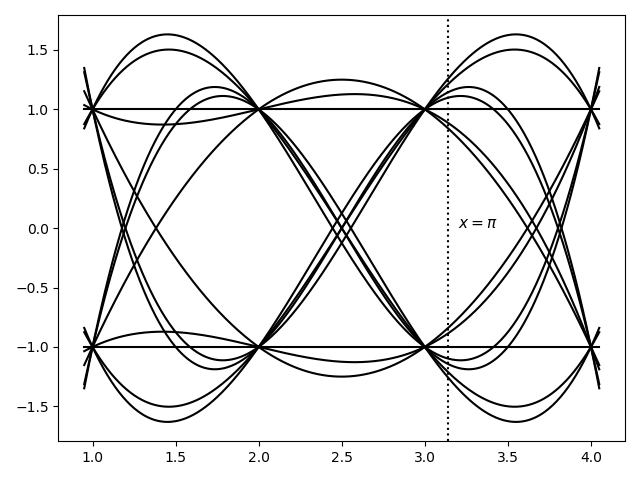}
    \caption{All possible 3-polynomials $p$ with $|p(x)| = 1$ for $x \in \{1, 2, 3, 4\}$.}
    \label{f:real:example}
\end{figure}
The polynomials whose value at $1$ is positive (i.e.~equal to $1$) are:
\begin{align*}
p_1(x) &= 1\\
p_2(x) &=-x^3 + 8x^2 - 19x + 13\\
p_3(x) &=x^3 - 7x^2 + 14x - 7\\
p_4(x) &=-\frac{1}{3}x^3 + 2x^2 - \frac{11}{3}x + 3\\
p_5(x) &=x^2 - 5x + 5\\
p_6(x) &=-\frac{4}{3}x^3 + 10x^2 - \frac{68}{3}x + 15\\
p_7(x) &=\frac{2}{3}x^3 - 5x^2 + \frac{31}{3}x - 5\\
p_8(x) &=-\frac{1}{3}x^3 + 3x^2 - \frac{26}{3}x + 7
\end{align*}
The remaining polynomials are their opposites. Observe that at $x = \pi$ no two non-opposite polynomials have the same absolute value, because $\pi$ is a transcendental number (i.e., it is not algebraic) and the coefficients of the above polynomials are rational. Therefore, we can choose $x = \pi$ for our final evaluation. 
\end{example}

The above example suggests that some information about coefficients of the polynomials can be used to find evaluation points in a non-adaptive way. 

\begin{theorem}
\label{t:non-adaptcountabler}
Suppose that $\mathbf{A}$ is a subfield of the field $\mathbf{R}$ of real numbers, such that the field extension $\mathbf{A}\subseteq \mathbf{R}$ is transcendental.  
The problem 
$$\sprobra$$ 
has a solution $(N,\phi)$ such that $N$ is a non-adaptive information operator using information 
operations from the class $|\Lambda^\textnormal{std}|$ and such that $\textnormal{cost}_i(N,v)=n+2$ for all 
$v\in \overline{\mathbf{A}_n[x]}$, and 
$\sup\{\textnormal{cost}_a(\phi,y)\colon y\in N(\overline{\mathbf{R}_n[x]})\}$ 
is exponential in $n$.
\end{theorem}
\begin{proof}
We proceed similarly as in the proof of Theorem~\ref{t:non-adaptreal}.

We will equivalently demonstrate that it is possible to identify a 
polynomial from $\mathbf{A}_n[x]$ by $n+1$ signless evaluations and one exact evaluation. However, this time all evaluation 
points (including the one to perform an exact evaluation) 
can be fixed beforehand and used without any need for adaption. 

Let us consider $p \in \mathbf{A}_n[x]$ and a sequence $x_0, x_1, \ldots, x_n \in \mathbf{A}$ of distinct elements of $\mathbf{A}$. 
For every tuple $\overline{b} \in \{-1, 1\}^{n+1}$ there is exactly one polynomial $w_{\overline{b}} \in \mathbf{A}_n[x]$ 
such that $w_{\overline{b}}(x_i) = b_i |p(x_i)|$ for all 
$0 \leq i \leq n$. This follows from two observations:
\begin{itemize}
\item $-1\mathbf{A} = \mathbf{A}$, since $\mathbf{A}$ is a field
\item a matrix with entries from $\mathbf{A}$ has an inverse over $\mathbf{A}$ iff it has an inverse over $\mathbf{R}$
\end{itemize}

Therefore, there are exactly $2^{n+1}$ polynomials from $\mathbf{A}_n[x]$ 
with values, up to a sign, matching 
those of $p$ in each of the points $x_0,x_1,\ldots,x_n$. 

From the assumption, the set: 
$$\catl{D}=\mathbf{R} \setminus \{x \in \mathbf{R} \colon \exists_{w\in\mathbf{A}_n[x]} \; w(x) = 0\}$$
is non-empty. By definition any $x \in \catl{D}$ differentiates every pair of distinct polynomials $w,v\in\mathbf{A}_n[x]$. Thus we may perform one exact evaluation at any fixed $x\in\catl{D}$, which will uniquely identify $\overline{b}$ such that $w_{\overline{b}}=p$, and hence also $p$.
\end{proof}

\begin{corollary}
The problem 
$$\sprobrq$$ 
can be solved using $n+1$ signless evaluations at any distinct points 
$x_0,x_1,\ldots,x_n\in\mathbf{Q}$, and one evaluation at any point 
$x\in\mathbf{R}$ such that $x$ is a transcendental number.
\end{corollary}

\begin{remark}
Under assumption of Theorem~\ref{t:non-adaptcountabler} one may solve $\spprob{\mathbf{A}}$ with a single exact evaluation at any transcendental (wrt.~$\mathbf{A}$) number $x$. This follows from the observation that the evaluation $\mor{L_x}{\pprob{\mathbf{A}}}{\mathbf{R}}$ at transcendental $x$ is an injection. On the other hand, without further assumptions about $\mathbf{A}$ there is no upper bound on the algorithmic cost of the problem (in fact, there may be no algorithm realizable on a Blum–Shub–Smale machine). 
\end{remark}


\section{Complex polynomials}
\label{s:complex}
Phaseless identification of complex polynomials is much more complicated than in the case of real polynomials. One reason for this increased difficulty is that the cardinality of the unit sphere for complex numbers, i.e. $\{x \in \mathbf{C} \colon |x| = 1\}$, is  continuum; whereas  the unit sphere for real numbers, i.e. $\{x \in \mathbf{R} \colon |x| = 1\}$, is just the 2-element set $\{-1, 1\}$. The next theorem exploits this difference.
\begin{theorem}\label{t:nocomplex}
There is no solution of the problem 
$$\spprobc$$ 
that relies on the information operator $N$ which 
uses at most $n$ information operations from the class $\Lambda^\textnormal{std}$ and at most $n+1$ information operations 
from the class $|\Lambda^\textnormal{std}|$.
\end{theorem}
\begin{proof}

Let us denote $\catl{B}_k = \{\overline{x} \in \mathbf{C}^k \colon \forall_{0 \leq i < k} \; |x_i| = 1\}$ and $\catl{B}_k^{(2)} = \{\tuple{\overline{x}, \overline{y}} \in \catl{B}_k \times \catl{B}_k \colon \overline{x} \neq \overline{y}\}$.
Consider a  a sequence 
$x_0, x_1, \ldots, x_n \in \mathbf{C}$ of distinct complex numbers and a polynomial $f \in \mathbf{C}_n[x]$ that does not have a root at any of these numbers --- i.e., $f(x_i) \neq 0$ for all  $0\leq i\leq n$.

Note that for every tuple $\overline{b} \in \mathbf{C}^{n+1}$ there is exactly one polynomial $w_{\overline{b}} \in \pprobc$ such that 
$w_{\overline{b}}(x_i) = \overline{b}_i |f(x_i)|$ 
for all $0\leq i\leq n$. 
By linearity, for all $\overline{b},\overline{b'}\in\mathbf{C}^{n+1}$, 
the condition $w_{\overline{b}} = w_{\overline{b'}}$ can be rewritten as $w_{\overline{b}- \overline{b'}} = 0$ and then as $w_{\overline{c}} = 0$, where $\overline{c} = \overline{b} - \overline{b'} $. Observe that the set  
$$\catl{U} =  \{\overline{b} - \overline{b'}\in\mathbf{C}^{n+1}\colon \tuple{\overline{b}, \overline{b'}} \in \catl{B}_{n+1}^{(2)}\}$$
is equal to the set
$$\{\overline{c} \in \mathbf{C}^{n+1} \colon \overline{c} \neq 0\ \land\   
|\overline{c}_i| \leq 2\ \textnormal{for all}\ 0\leq i\leq n\}$$
Therefore, if we define
$$\catl{S}=\{\overline{x'} \in \mathbf{C}^{n} \colon 
\exists_{\overline{c} \in \catl{U}} \;  w_{\overline{c}}(\overline{x'}_i) = 0 
\ \textnormal{for all}\ 0\leq i\leq n-1\}$$ 
we see that for $\overline{x'} \in \mathbf{C}^{n}$ we have: 
$$\overline{x'} \in \catl{S} \Leftrightarrow 
\exists_{\overline{b}, \overline{b'}\in\catl{B}_{n+1}^{(2)}} \forall_{0 \leq i < n} \;  w_{\overline{b}}(\overline{x'}_i) = w_{\overline{b'}}(\overline{x'}_i) 
$$ 
We argue that $\catl{S}=\mathbf{C}^n$. Consider any tuple 
$\overline{x'} \in \mathbf{C}^n$ and any non-zero polynomial $p \in \pprobc$ 
for which $\overline{x'}$ is a tuple of all its roots. 
If $M = \frac{1}{2}\max_{0\leq i\leq n}{\frac{|p(\overline{x}_i)|}{|f(\overline{x}_i)|}}$, 
then the polynomial $p/M$ is non-zero and satisfies 
$$\frac{|(p/M)(\overline{x}_i)|}{|f(\overline{x}_i)|} \leq 2\ \textnormal{for all}
\ 0\leq i\leq n.$$ 
Therefore, there is a non-zero tuple 
$\overline{c}\in\mathbf{C}^{n+1}$ such that 
$(p/M)(\overline{x}_i) = \overline{c}_i |f(\overline{x}_i)|$ and $|\overline{c}_i| \leq 2$ 
for all $0\leq i\leq n$. Since $p/M \in \mathbf{C}_n[x]$, it follows that 
$p/M=w_{\overline{c}}$, and hence $w_{\overline{c}}(\overline{x'}_i) = 0$ 
for all $0\leq i\leq n-1$. It means that $\overline{x'}\in\catl{S}$, so 
indeed $\catl{S}=\mathbf{C}^n$. 

Note that $f \in \{w_{\overline{b}}\in\mathbf{C}_n[x] \colon \overline{b}\in 
\catl{B}^{n+1}\}$ and for every $w_{\overline{b}}$, it  holds that 
$$|w_{\overline{b}}(x_i)|=|f(x_i)|\ \textnormal{for all}\ 0\leq i\leq n$$
Now, if $x'_0,x'_1,\ldots,x'_{n-1}\in\mathbf{C}$ is any sequence of distinct complex numbers,  
there are some $\overline{b},\overline{b'}\in\catl{B}^{n+1}$ such that 
$\overline{b} \neq \overline{b'}$ and 
$w_{\overline{b}}(x'_i) = w_{\overline{b'}}(x'_i)= 
f(x'_i)$ 
for all $0\leq i \leq n-1$. Since $f$ is equal to at most one of 
$w_{\overline{b}}$ or $w_{\overline{b'}}$, we cannot identify $f$. 
\end{proof}

\begin{corollary}
There is no solution of the problem 
$$\sprobc$$ 
that relies on the information operator $N$ using information operations from 
the class $|\Lambda^\textnormal{std}|$ and such that 
$\textnormal{cost}_i(N,v)\leq 2n+1$ for all $v\in \overline{\mathbf{C}_n[x]}$.
\end{corollary}
\begin{proof}
Since the information $|\Lambda^{\textnormal{std}}|$ is weaker than the information 
$\Lambda^{\textnormal{std}}$, it follows from Theorem~\ref{t:nocomplex} that 
$\spprobc$ cannot have a solution based on an information operator using exactly 
one information operation from the class $\Lambda^\textnormal{std}$ and $2n$ 
information operations from the class $|\Lambda^\textnormal{std}|$. Consequently, 
$\sprobc$ cannot have a solution based on an information operator using $2n+1$ 
information operations from the class $|\Lambda^\textnormal{std}|$.
\end{proof}

Let us see how Theorem~\ref{t:nocomplex} works on a simple example.
\begin{example}\label{e:nogo:reals}
Let $f \in \probc$ and assume that the evaluation points are $\overline{x} = \tuple{1, 2, 3, 4}$ with their corresponding values all equal to $1$. Consider (any) triple of points, for example $\{0, 5, 6\}$ and take any non-zero polynomial whose roots are at these points, for instance:
$$p(x) = \frac{1}{24}x^3 - \frac{11}{24} x^2 + \frac{5}{4} x$$
The polynomial $p$ maps tuple $\overline{x} = \tuple{1, 2, 3, 4}$ to tuple $\overline{c} = \tuple{\frac{5}{6}, 1, \frac{3}{4}, \frac{1}{3}}$. Therefore, $p(x_i)=c_i f(x_i)$ for all $0 \leq i < 4$ and $|c_i|\leq 2$ for $0\leq i < 4$. Let us decompose $\overline{c}$ as $\overline{c} = \overline{b} - \overline{b}'$, where 
$| b_i| = |b_i'| = 1$ for $0\leq i \leq 3$:
\begin{align*}
\overline{b} &= \tuple{\frac{5 + \sqrt{119}i}{12},\frac{1+\sqrt{3}i}{2},\frac{3+\sqrt{55}i}{8},\frac{1+\sqrt{35}i}{6}}\\
\overline{b}' &= \tuple{\frac{-5 + \sqrt{119}i}{12},\frac{-1+\sqrt{3}i}{2},\frac{-3+\sqrt{55}i}{8},\frac{-1+\sqrt{35}i}{6}}
\end{align*}

Figure~\ref{f:complex:example:wb} shows interpolating polynomials $w_{\overline{b}}$ and $w_{\overline{b}'}$, whereas  Figure~\ref{f:complex:example:abs} shows the plot of $|w_{\overline{b}}|$, which is equal to $|w_{\overline{b}'}|$ on $\mathbf{R}$.

 \begin{figure}
  \centering
    \includegraphics[width=\textwidth]{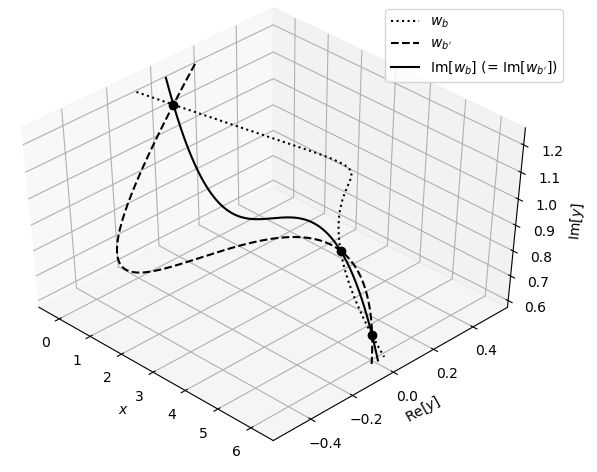}
    \caption{Plot of $w_b$ and $w_{b'}$ for $x \in [-0.1, 6.1]$.}
    \label{f:complex:example:wb}
\end{figure}

\begin{figure}
  \centering
    \includegraphics[width=0.7\textwidth]{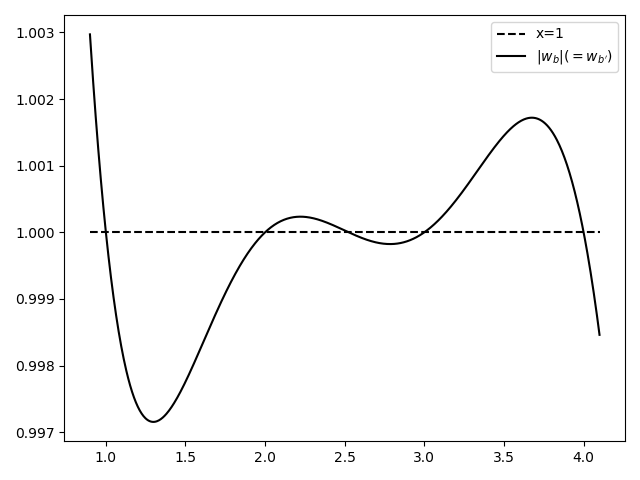}
    \caption{Plot of $|w_b|$ on $\mathbf{R}$ (which is equal to $ |w_{b'}|$ on $\mathbf{R}$).}
    \label{f:complex:example:abs}
\end{figure}
Note that the absolute values of both $w_{\overline{b}}$ and $w_{\overline{b}'}$ are all ones at $\tuple{1, 2, 3, 4}$, and the values are equal on  $\tuple{0, 5, 6}$ 
(in fact, in this example, the absolute values are equal on the whole real line).
\end{example}

The above example shows that if we are not careful enough in choosing the evaluation points, we may not be able to distinguish between the polynomials in \emph{any} number of evaluations (recall Figure~\ref{f:complex:example:abs}). The next example shows that it is not sufficient to evaluate polynomials at points with different imaginary values.

\begin{example}\label{e:nogo:units}
Consider the following polynomials:
\begin{align*}
p(x) &= 2x + 1\\
q(x) &= x + 2\\
\end{align*}
We have:
\begin{align*}
|p(e^{ix})|^2 &= 2e^{ix} + 1 = (2\cos(x) + 1)^2 + 4\sin^2(x) = 4\cos(x) + 5\\
|q(e^{ix})|^2 &= e^{ix} + 2 = (\cos(x) + 2)^2 + \sin^2(x) = 4\cos(x) + 5 \\
\end{align*}

Therefore, $|p(e^{ix})| = |q(e^{ix})|$, i.e.~the absolute values of $p$ and $q$ on the unit circle are equal.
\end{example}

On the other hand, if we know the polynomial in advance, we may always choose $n+2$ points such that the absolute values at the chosen points identify the polynomial up to a phase.

\begin{example}
Let $\omega_0, \omega_1,\ldots, \omega_{n}$ be the sequence of $(n+1)$-th primitive roots of $1$. There is exactly one polynomial $w$ of degree at most $n$ such that: $w(0) = 1$ and $|w(\omega_i)| = 1$ for all $0\leq i \leq n$.

Consider a polynomial $w(x) = a_0 + a_1x + a_2x^2 + \cdots + a_nx^n$. The constraint $w(0) = 1$ says that $a_0 = 1$. Now, let us take the Vandermonde matrix $T$ at $\omega_i$. We have to solve the equation: $T \overline{a} = \overline{u}$, where $u_i = 1$ and $a_0 = 1$. Because $T$ is non-singular, this is equivalent to the equation $\overline{a} = T^{-1} \overline{u}$. The first row of the inverse of the Vandermonde matrix is of the form $[\frac1{n+1}, \frac1{n+1}, \cdots, \frac1{n+1}]$. Thus we obtain equation: $1 = \frac1{n+1} \sum_{i=0}^{n} u_i$, which is satisfied only if $u_i = 1$ for all $0 \leq i \leq n$. 
Therefore, $\overline{a} = T^{-1} \overline{1}$ is the unique solution.
\end{example}

Example~\ref{e:nogo:reals} shows that it is not sufficient to evaluate polynomials at any fixed points with the same \emph{imaginary} value, whereas Example~\ref{e:nogo:units} shows that it is not sufficient to evaluate polynomials at any fixed points with the same \emph{real} value. One may wonder if it is sufficient to evaluate polynomials at some fixed set of points whose both \emph{real} and \emph{imaginary} values vary. The next theorem gives an affirmative answer to this question. 

\begin{theorem}
The computational complexity of the problem 
$$\sprobc$$ 
for the class $|\Lambda^{\textnormal{std}}|$ is polynomial in $n$. 
Furthermore, it has a solution $(N,\phi)$ such that $N$ is nonadaptive and $\textnormal{cost}_i(N,v)=(2n+1)^2$ for all 
$v\in \overline{\mathbf{C}[x]}$, and  $\sup\{\textnormal{cost}_a(\phi,y)\colon y\in N(\overline{\mathbf{C}_n[x]})\}$ 
is polynomial in $n$.
\end{theorem}
\begin{proof}
Let us consider the general form of a complex polynomial ${p \in \pprobc}$:
$$p(x) = a_0e^{\alpha_0i} + a_1e^{\alpha_1i}x + a_2e^{\alpha_2i}x^2 + \cdots + a_ne^{\alpha_ni}x^n$$
where all parameters are real and $a_i$ are non-negative. By using the circular coordinates for $x$, we get:
$$p(xe^{yi}) = a_0e^{\alpha_0i} + a_1xe^{(\alpha_1 + y)i} + a_2x^2e^{(\alpha_2 + 2y)i} + \cdots + a_nx^ne^{(\alpha_n + ny)i}$$
The formula for the (square of) absolute value of $p(xe^{yi})$ takes the following form:
\begin{align*}
|p(xe^{yi})|^2 &= (a_j\cos(\alpha_0) + a_{1}x\cos(\alpha_{1} + y) + \cdots + a_nx^{n}\cos(\alpha_n + ny))^2\\
&+ (a_j\sin(\alpha_0) + a_{1}x\sin(\alpha_{1} + y) + \cdots + a_nx^{n}\sin(\alpha_n + ny))^2\\
&= \sum_{m = 0}^n \sum_{k=m}^n b_{k, m} x^{2k-m}(\cos(\beta_{k, m}) \cos(my) + \sin(\beta_{k, m})\sin(my))\\
\end{align*}
where:

\begin{align}\label{eq:new:coeff}
b_{k, m} &= \left\{
  \begin{array}{ll}
    a_k^2 & \textnormal{if}\  m = 0  \\
    2 a_k a_{k-m} & \textnormal{otherwise} \\
  \end{array} \right.\\
\beta_{k, m} &= \alpha_k - \alpha_{k-m}
\end{align}

Let $x_0, x_1, x_2, \cdots, x_{2n}$ be a sequence of distinct non-negative real numbers. For every $x_i$ let us define function $\mor{\widehat{x_i}}{[-\pi, \pi]}{\mathbf{R}}$ as follows:
$$\widehat{x_i}(\alpha) = |p(x_ie^{\alpha i})|^2$$
By the above observation, $\widehat{x_i}$ is a trigonometric polynomial:
$$\widehat{x_i}(\alpha) = \sum_{m=0}^{n} A_{x_i, m} \cos(m \alpha) + B_{x_i, m} \sin(m \alpha)$$
with:
\begin{align*}
A_{x_i, 0} &= \sum_{k=j}^n b_{k, 0}x_i^{2k}\\
A_{x_i, m} &= \sum_{k=m}^n  b_{k, m} x_i^{2k-m}\cos(\beta_{k, m})\\
B_{x_i, 0} &= 0\\
B_{x_i, m} &= \sum_{k=m}^n  b_{k, m} x_i^{2k-m}\sin(\beta_{k, m})\\
\end{align*}
To recover coefficients $A_{x_i, m}$ and $B_{x_i, m}$ we use a standard trick. If we set:

$$c_{x_i}(\alpha) = \frac{\widehat{x_i}(\alpha) + \widehat{x_i}(-\alpha)}{2} = \sum_{m=0}^{n} A_{x_i, m} \cos(m \alpha)$$
the coefficients $A_{x_i, m}$ may be recovered by the inverse cosine transform from values $c_{x_i}(\omega_0), c_{x_i}(\omega_1), c_{x_i}(\omega_2), \dotsc, c_{x_i}(\omega_{2n})$ where $\omega_k = \frac{k\pi}{n}$. Similarly, setting:
$$s_{x_i}(\alpha) = \frac{\widehat{x_i}(\alpha) - \widehat{x_i}(-\alpha)}{2} = \sum_{m=1}^{n} B_{x_i, m} \sin(m \alpha)$$
gives the coefficients $B_{x_i, m}$ by the inverse sine transform from values $s_{x_i}(\omega_1), s_{x_i}(\omega_2), \dotsc, s_{x_i}(\omega_{2n})$.

Observe, that to recover both $A_{x_i, m}$ and $B_{x_i, m}$ we need $2n+1$ evaluations of $\widehat{x_i}$. Therefore, to recover $A_{x_i, m}$ and $B_{x_i, m}$ at every $x_i$ for $0 \leq i \leq 2n$, we need $(2n+1)^2$ evaluations of $|p|$.
 
Now, let us consider the following polynomials of order at most $2n$:
\begin{align*}
A_0(x) &= \sum_{k=j}^n b_{k, 0}x^{2k}\\
A_m(x) &= \sum_{k=m}^n  b_{k, m} x^{2k-m}\cos(\beta_{k, m})\\
B_m(x) &= 0\\
B_m(x) &= \sum_{k=m}^n  b_{k, m} x^{2k-m}\sin(\beta_{k, m})\\
\end{align*}
By definition $A_m(x_i) = A_{x_i, m}$ and $B_m(x_i) = B_{x_i, m}$, so we know the values of each $A_m, B_m$ at $2n+1$ distinct points. Therefore, we may recover coefficients $b_{k, 0}$, $b_{k, m}\cos(\beta_{k, m})$ and $b_{k, m} \sin(\beta_{k, m})$ by polynomial interpolation.

Because $a_k$ are non-negative, they are uniquely determined by $a_k^2 = b_{k, 0}$. From Equation~\ref{eq:new:coeff} we can compute remaining $b_{k, m}$. Observe, that now we may compute the values of each pair $\tuple{\cos(\beta_{k, m}), \sin(\beta_{k, m})}$ and then: $\beta_{k, m} = \alpha_k - \alpha_{k-m}$. Fixing $\alpha_0$ fixes the rest of the angles, therefore up to $\alpha_0$ all coefficients of polynomial $p$ are uniquely determined. 
\end{proof}

\section{Summary and further work}
\label{s:conclusion}
In the present paper we have investigated the relation between the number of (exact and 
phaseless) polynomial evaluations and combinatorial cost needed to process it in order 
to recover an unknown polynomial. However, some questions still remain open. Note that 
in Theorem~\ref{t:adaptreal} we have shown that $n+2$ (phaseless) polynomial evaluations 
are sufficient for the (phaseless) recovery of a polynomial, however the combinatorial 
cost of processing of that information, being in fact the cost of a full search 
among all possible 
candidates for a solution, is exponential in~$n$. Hence the natural question arises: 
can we do it faster? More generally, what is the minimal number of (phaseless) evaluations 
of an unknown polynomial $p$ such that it is possible to perform (phaseless)  identification of $p$ in a polynomial time? It is of interest to investigate this both 
in the real and the complex case.

\vskip 1em
\begin{acknnowledgments}
This research was supported by the National Science Centre, Poland, under 
projects 2018/28/C/ST6/00417 (M.~R.~Przyby{\l}ek) and 2017/25/B/ST1/00945 
(P.~Siedlecki).
\end{acknnowledgments}


\bibliographystyle{apalike}
\bibliography{bibIBC}

\end{document}